\documentclass[letterpaper, 10 pt, conference]{ieeeconf}

\IEEEoverridecommandlockouts
\usepackage{cite}

\usepackage{amsmath,amsthm,amssymb,amsfonts}
\allowdisplaybreaks
\usepackage{algorithmic}
\usepackage{graphicx}
\usepackage{textcomp}
\usepackage{xcolor}
\usepackage{bbm}
\usepackage{balance}
\usepackage{bm}
\usepackage{algorithmic}
\usepackage{algorithm}
\usepackage[caption=false, font=footnotesize]{subfig}
\newtheorem{theorem}{Theorem}

\newtheorem{lemma}{Lemma}
\newtheorem{assumption}{Assumption}
\newtheorem{remark}{Remark}

\ifodd 1

\newcommand{\wjcom}[1]{\textbf{\color{white!50!purple} (Wenjie comment: #1)}} 
\else

\newcommand{\wjcom}[1]{}
\fi

\newcommand{\wwb}[1]{\textcolor{black}{#1}}

\def\BibTeX{{\rm B\kern-.05em{\sc i\kern-.025em b}\kern-.08em
    T\kern-.1667em\lower.7ex\hbox{E}\kern-.125emX}}
\begin{document}


\title{\LARGE \bf Human-in-the-loop: Real-time Preference Optimization
\thanks{
The authors are with Automatic Control Laboratory, EPFL, Switzerland. Email: wenbin.wang@epfl.ch, wenjie.xu@epfl.ch, colin.jones@epfl.ch. This work was supported by the Swiss National Science Foundation under NCCR Automation, grant agreement {51NF40\_180545} and the Swiss Federal Office of Energy SFOE as part of the SWEET consortium SWICE.}
}

\newcommand{\obj}{\Tilde{\Phi}}
\newcommand{\gest}{\Tilde{\nabla}}
\newcommand{\tr}[1]{\text{tr}(#1)}
\newcommand{\E}[1]{\bb{E}[#1]}
\newcommand{\Ef}[2]{\bb{E}[#1|\mathcal{F}_{#2}]}
\newcommand{\N}{\frac{1}{N}\mathbf{1}\mathbf{1}^{\top}}
\newcommand{\inff}[1]{#1_{\infty}}
\newcommand{\bb}[1]{\mathbb{#1}}
\newcommand{\norm}[1]{\|#1\|}

\author{Wenbin Wang, Wenjie Xu, Colin N. Jones}

\maketitle

\begin{abstract}

Optimization with preference feedback is an active research area with many applications in engineering systems where humans play a central role, such as building control and autonomous vehicles. While most existing studies focus on optimizing a static user utility, few have investigated its closed-loop behavior that accounts for system transients. In this work, we propose an online feedback optimization controller that optimizes user utility using pairwise comparison feedback with both optimality and closed-loop stability guarantees. By adding a random exploration signal, the controller estimates the descent direction based on the binary comparison feedback between two consecutive time steps. We analyze its closed-loop behavior when interacting with a nonlinear plant and show that, under mild assumptions, the controller converges to the optimal point without inducing instability. Theoretical findings are further validated through numerical experiments.
\end{abstract}
\section{Introduction}
Humans are the key components in engineering systems and the primary beneficiaries of many leading technologies, such as building control \cite{eichler2018humans} and human-robot collaboration \cite{ajoudani2018progress}. It is essential to design a human-aware controller capable of regulating the system in real time to optimize user latent utility. Existing controllers typically track a predefined reference, which is often derived from large population models, e.g., indoor room temperature for building control \cite{fanger1970thermal}. While being simple and easy to implement, this can introduce bias and lead to suboptimal performance, as it fails to account for individual differences. Moreover, without real-time human feedback, such a controller cannot respond to time-varying utility and is not robust to external disturbances.

As an emerging real-time control technique, Online Feedback Optimization (OFO) \cite{hauswirth2024optimization} has been effective in applications such as grid control \cite{ortmann2020experimental} and robot coordination \cite{terpin2022distributed}. By taking real-time system output, OFO can navigate the system toward its optimal point without requiring precise knowledge of the plant dynamics or measurement noise. Theoretical guarantees have been established for both first-order \cite{8673636} and zeroth-order \cite{he2023model} formulations. However, people are generally more adept at making relative comparisons than providing absolute evaluations of their utility  \cite{kahneman2013prospect}. As a result, human feedback often appears as pairwise comparisons, making it difficult to directly apply existing OFO schemes to design human-aware controllers with closed-loop guarantees.

Research on offline optimization with preference feedback has been rapidly expanding. In the finite action setting, preferences can be encoded in a preference matrix, where each entry gives the probability that one option is preferred over another. Notions of optimality such as the Copeland winner \cite{zoghi2015copeland} and the Borda winner \cite{urvoy2013generic} have been studied, and algorithms based on random exploration \cite{zoghi2014relative} are employed to identify the optimal choice. For continuous action spaces, a common assumption is that preferences are induced by a latent utility, which is often modeled with Gaussian Processes (GP). Heuristic strategies for sequential decision-making balance exploration and exploitation \cite{gonzalez2017preferential,Previtali_2023}, while regret guarantees can be established under assumptions such as when the utility lies in a Reproducing Kernel Hilbert Space (RKHS) \cite{xu2024principled}. In parallel, gradient-estimation methods from preference feedback have been explored \cite{yue2009interactively,saha2021dueling}, with optimality guarantees under assumptions such as smoothness and convexity \cite{saha2021dueling}. 

Despite recent advancements, most existing work considers a static problem and neglects transient system behavior. Many open challenges remain in developing online preference optimization algorithms that account for system transients. First, the stability of the closed-loop system is hard to quantify, as the binary nature of preference feedback renders the overall dynamics highly nonlinear. Moreover, people typically have limited knowledge about the plant dynamics, implying that directly following their preferences can drive the system toward instability. Second, tracking the optimal point in real time with a controller trained on offline data is difficult, as the utility function can be time-varying, subject to external disturbances, and environment-dependent \cite{kahneman2013prospect}. The presentation of alternatives can also shape individuals’ expressed preference, given that people are not always rational \cite{kahneman2013prospect}. A carefully designed mechanism for how humans interact with the controller is necessary to ensure stable and efficient system operation.

Inspired by the recently proposed model-free OFO with one-point residual estimation \cite{he2023model}, we introduce a novel OFO controller that leverages binary preference feedback to optimize user utility while ensuring closed-loop stability. To the best of our knowledge, it is the first work addressing the real-time preference optimization problem with closed-loop guarantees. Our approach employs a stochastic scheme, in which a random exploration signal is added into the system at each time step. Preference feedback is then collected based on the perturbed utilities between consecutive time steps. Unlike existing approaches \cite{yue2009interactively,saha2021dueling} that require two function evaluations at each time step, our method requires only one evaluation, making it well-suited for online implementation. Under mild assumptions, we show that the resulting update imitates a gradient descent step and provide theoretical guarantees on the stability and optimality of the closed-loop system. Theoretical results are further supported through numerical experiments on a thermal comfort optimization problem.
\section{Problem Formulation and Preliminaries}
\subsection{Notation}
Let $\mathbb{R}^n$ be the $n$-dimensional Euclidean space and $\mathbb{N}$ be the set of natural numbers. For a real-valued function $f:\mathbb{R}^n\rightarrow \mathbb{R}$, we denote its gradient at $x\in \mathbb{R}^{n}$ by $\nabla f(x)$. For a single input function $\sigma:\mathbb{R}\rightarrow \mathbb{R}$, we denote its derivative by $\sigma'(t)$. The realization of vector $x\in\mathbb{R}^n$ at time step $k$ is written as $x_k$.

\subsection{Problem Formulation}
Consider an exponentially stable plant with $n_x\in\mathbb{N}$ states and $n_u\in\mathbb{N}$ inputs
\begin{equation}
\label{eq: nonlinear system}
    x_{k+1} = f(x_k,u_k),
\end{equation}
where $f:\mathbb{R}^{n_x}\times\mathbb{R}^{n_u}\rightarrow \mathbb{R}^{n_x}$ is the state transition function, $x\in \mathbb{R}^{n_x}$ is the plant state, $u_k\in \mathbb{R}^{n_u}$ is the plant input. We assume that there exists a unique steady-state input-state map $h:\mathbb{R}^{n_u}\rightarrow \mathbb{R}^{n_x}$ such that $\forall u \in \mathbb{R}^{n_u}$, $h(u)= f(h(u),u)$.
\begin{assumption}
\label{ap: map lipschitz}
    The input-state map $h$ is $L_{h}$-Lipschitz continuous.
\end{assumption}
The assumption of an exponentially stable plant with Lipschitz continuous input-state map is satisfied for many systems, e.g., linear system $x_{k+1} = Ax_k+Bu_k$ with $A$ being stable. The results of this work can also be extended to include the output measurement $y_k=g(x_k,u_k)$, where additional assumptions, such as a Lipschitz input-output map, are necessary.

According to the converse Lyapunov theorem \cite[p. 194]{he2023model,khalil2002nonlinear}, for an exponentially stable plant, there exists a Lyapunov function $V:\mathbb{R}^{n_x}\times\mathbb{R}^{n_u}\rightarrow\mathbb{R}$ and positive parameters $\alpha_1,\alpha_2,\alpha_3$ such that for each fixed $u\in \mathbb{R}^{n_u}$,
\begin{align}
    &\alpha_1\|x-h(u)\|^2\leq V(x,u)\leq \alpha_2\|x-h(u)\|^2,\\
    &V(f(x,u),u)-V(x,u)\leq-\alpha_3\|x-h(u)\|^2.
\end{align}
The parameter $\alpha_3$ characterizes how fast the system stabilizes. A larger $\alpha_3$ indicates the system stabilizes to $h(u)$ with a larger rate.

In this work, we want to design a optimal controller such that \eqref{eq: nonlinear system} is stabilizing to the solution of the following optimization problem
\begin{equation}
\label{eq: constrained objective}
    \begin{split}
    \min_{x,u} \quad&\Phi(x,u),\\
    \text{s.t.} \quad& x = h(u),
    \end{split}
\end{equation}
where $\Phi:\mathbb{R}^{n_x}\times\mathbb{R}^{n_u}\rightarrow\mathbb{R}$ is the latent utility function. This utility function depends on both the system state and input, capturing real-world scenarios where these factors jointly have influences. For example, in building energy management, the objective often consists of both thermal comfort and control cost \cite{eichler2018humans}.
\begin{assumption}
\label{ap: objective lipschitz}
    $\forall u\in \mathbb{R}^{n_u}$, the latent utility $\Phi(x,u)$ is $L_{x}$-Lipschitz with respect to the first argument, i.e., $\|\Phi(x_1,u)-\Phi(x_2,u)\|\leq L_x\|x_1-x_2\|, \forall x_1, x_2 \in \mathbb{R}^{n_x}$, where $L_{x}$ is a uniform Lipschitz constant.
\end{assumption}

Problem \eqref{eq: constrained objective} can be reformulated as an unconstrained problem by replacing $x$ with $x=h(u)$ in $\Phi(x,u)$, i.e.,
\begin{equation}
\label{eq: unconstrained objective}
\min_{u} \tilde{\Phi}(u),
\end{equation}
where $\tilde{\Phi}(u)=\Phi(h(u),u)$.
\begin{assumption}
\label{ap: convex}
    The latent function $\tilde{\Phi}(u)$ is $L_{0}$-Lipschitz, $L_1$-smooth and $m$-strongly convex.
\end{assumption}
Assumption \ref{ap: convex} is commonly employed for theoretical analysis \cite{simpson2020analysis,wang2024online,wang2024decentralized} and is satisfied in applications, e.g., building control \cite{eichler2018humans} and frequency regulation \cite{zhao2015distributed}. \wwb{We denote the solution of \eqref{eq: unconstrained objective} by $u^*$. The assumption on strongly convexity assures that $u^*$ is unique.}

Classical numerical optimization methods, such as gradient descent, require the knowledge of $h$ to find the optimal solution \cite{hauswirth2024optimization}. However, this is impractical for systems where building a high-fidelity model is costly. To address this limitation, zeroth-order feedback optimization methods have emerged as a promising approach. It estimates the gradient using finite differences, either by considering plant transients \cite{he2023model} or assuming an algebraic plant \cite{wang2024online}. 

Nevertheless, finite difference fails with preference feedback since the true utility $\Phi(x,u)$ cannot be directly queried. It often appears in the form of pairwise comparisons, i.e., binary signals. To bridge this gap, we need to link the binary feedback with the utility through user models.

\subsection{User model}
 Given $u_1$ and $u_2$, we denote the event `$\tilde{\Phi}(u_1)<\tilde{\Phi}(u_2)$' by $u_1\succ u_2$, i.e., $u_1$ is better than $u_2$. The corresponding preference feedback is written as $\bm{1}_{u_1\succ u_2}$, which is
\begin{equation}
    \bm{1}_{u_1\succ u_2} = 
    \begin{cases}
    1, &\text{if $u_1$ is better},\\
    -1, &\text{if $u_2$ is better}.
\end{cases}
\end{equation}
To link the utility $\tilde{\Phi}$ with preference feedback, we adopt a probabilistic model as shown in Assumption \ref{ap: user model}.
\begin{assumption}
\label{ap: user model}
    The preference feedback $\bm{1}_{u_1\succ u_2}$ follows a Bernoulli distribution, i.e., $\mathbb{P}(\bm{1}_{u_1\succ u_2}=1)=\sigma(\tilde{\Phi}(u_2)-\tilde{\Phi}(u_1))$, where $\sigma(t) = \frac{1}{1+e^{-t}}$.
\end{assumption}
This model, i.e., the Bradley-Terry Model, among other probabilistic models such as the Thurstone-Mosteller model \cite[Section 2.2.3]{mikkola2024humans}, are commonly found in the literature for preferential learning \cite{yue2009interactively,xu2024principled}. The idea is that when $\tilde{\Phi}(u_1)$ is small, the probability of reporting $\bm{1}_{u_1\succ u_2}=1$ is large. We denote the Lipschitz constant and smoothness constant of $\sigma(t)$ by $L_{\sigma,0}$, $L_{\sigma,1}$, respectively. The result in this work also holds with other linking functions, provided they are monotonically increasing, rotation-symmetric, satisfy $\sigma(-\infty)=0$, $\sigma(\infty)=1$, and exhibit convexity for $x\leq 0$, essentially behaving like cumulative distribution functions \cite{yue2009interactively}.

\section{Online gradient estimation with preference feedback}
\label{section: 3}


We build on the existing result in model-free feedback optimization \cite{he2023model} and design a controller that estimates the descent direction with preference feedback from real-time state measurement. Different from existing preference optimization methods, our formulation considers the system transients explicitly. We assume that the binary preference feedback is collected from the user between two consecutive utility evaluations. The controller then regulates the system toward the point where the utility is minimized. This is summarized in Algorithm \ref{algorithm}.

\begin{algorithm}[!t]
    \renewcommand{\algorithmicrequire}{\textbf{Input:}}
    \renewcommand{\algorithmicensure}{\textbf{Output:}}
    \caption{Controller with comparison feedback}
    \label{algorithm}
    \begin{algorithmic}[1]
        \REQUIRE step size $\eta$, smoothing parameter $\delta$, $u_0 \in \bb{R}^{n_u}$, number of time steps $T$
        \FOR{$k=1, \ldots, T-1$}
            \STATE $x_{k+1} = f(x_k,u_k+\delta v_{k})$; \label{algorithm: 1}
            \STATE ask users to express their preference between $\Phi(x_{k+1},u_k+\delta v_k)$ and $\Phi(x_{k},u_{k-1}+\delta v_{k-1})$, sample $\bm{1}_{(x_{k+1},u_k+\delta v_{k})\succ (x_{k},u_{k-1}+\delta v_{k-1})}$;\label{algorithm: 2}
            \STATE update the input via \\$u_{k+1} = u_k + \frac{\eta}{2\delta} \bm{1}_{(x_{k+1},u_k+\delta v_{k})\succ (x_{k},u_{k-1}+\delta v_{k-1})} v_k$;\label{algorithm: 3}
        \ENDFOR
        \ENSURE $u_T$
    \end{algorithmic}
\end{algorithm}

At time step $k$, a random exploration signal $v_k$ is added to the system as shown in Line $\ref{algorithm: 1}$ of Algorithm \ref{algorithm}. It is drawn independent and identically distributed from the $(n_u-1)$-dimensional unit sphere $\mathbb{S}^{n_u-1}$ uniformly. We then use $x_{k+1}$ as an approximation of $h(u_k+\delta v_k)$, which allows us to approximate $\tilde{\Phi}(u_k+\delta v_k)$ with $\Phi(x_{k+1},u_k+\delta v_k)$.

In Line \ref{algorithm: 2}, the user is asked to provide a preference between $\Phi(x_{k+1},u_k+\delta v_k)$ and $\Phi(x_{k},u_{k-1}+\delta v_{k-1})$, i.e., sampling the random variable $\bm{1}_{(x_{k+1},u_k+\delta v_{k})\succ (x_{k},u_{k-1}+\delta v_{k-1})}$. We slightly abuse the notation here, denoting the event `$\Phi(x_{k+1},u_k+\delta v_k)<\Phi(x_{k},u_{k-1}+\delta v_{k-1})$' by $(x_{k+1},u_k+\delta v_{k})\succ (x_{k},u_{k-1}+\delta v_{k-1})$ and representing the preference feedback via the random variable $\bm{1}_{(x_{k+1},u_k+\delta v_{k})\succ (x_{k},u_{k-1}+\delta v_{k-1})}$ with the probability given by $\mathbb{P}(\bm{1}_{(x_{k+1},u_k+\delta v_{k})\succ (x_{k},u_{k-1}+\delta v_{k-1})}=1) = \sigma(\Phi(x_{k},u_{k-1}+\delta v_{k-1})-\Phi(x_{k+1},u_k+\delta v_k))$.

Finally, the control input $u_k$ is updated as shown in Line \ref{algorithm: 3}. When $\Phi(x_{k+1},u_k+\delta v_k)$ has a lower value, $\mathbb{P}(\bm{1}_{(x_{k+1},u_k+\delta v_{k})\succ (x_{k},u_{k-1}+\delta v_{k-1})}=1)$ is higher, implying that, with high probability, the algorithm updates $u_{k+1}$ toward $u_k+\delta v_k$, where the utility attains a smaller value.

The resulting closed-loop system takes the form
\begin{align}
\label{eq: algorithm 1}
    x_{k+1} = &f(x_k,u_k+\delta v_{k}),\\
    \label{eq: algorithm 2}
    u_{k+1} = &u_k + \frac{\eta}{2\delta} \bm{1}_{(x_{k+1},u_k+\delta v_{k})\succ (x_{k},u_{k-1}+\delta v_{k-1})} v_k.
\end{align}
The random perturbation $v_k$ facilitates exploration of multiple directions at each time step. The collected preference feedback indicates the descent direction of the unknown utility, effectively imitating a stochastic gradient descent step. Although \eqref{eq: algorithm 2} does not explicitly estimate $\nabla \tilde{\Phi}(u)$ since the preference feedback lacks information about utility values, we will demonstrate that \eqref{eq: algorithm 2} effectively performs gradient descent on a probability function, ultimately converging toward the optimal point in Section \ref{section: 4}.
\section{Performance Analysis}
\label{section: 4}
We first show that the closed-loop system is stable.
\begin{lemma}
\label{lemma: lyapunov stable}
    Let Assumptions \ref{ap: map lipschitz}-\ref{ap: user model} hold. $\mathbb{E}[V(x_k,u_k+\delta v_k)]\leq \mu^{k}\mathbb{E}[V(x_0,u_0+\delta v_0)]+\frac{a_1}{1-\mu}(2\delta^2+\eta+(\frac{\eta}{2\delta})^2)$, where $\mu = \frac{2\alpha_2}{\alpha_1}(1-\frac{\alpha_3}{\alpha_2})$, and $a_1 = 4\alpha_2L_{h}^2$.
\end{lemma}
\begin{proof}
    Proof can be found in Appendix \ref{proof: lyapunov stable}.
\end{proof}
\begin{remark}
    We consider the expected Lyapunov function $\mathbb{E}[V(x_k,u_k+\delta v_k)]$ as a stability indicator. The decay rate $\mu$ is a function of $\alpha_1,\alpha_2,\alpha_3$, which characterizes how quickly the system stabilizes to an equilibrium state under a constant input. A smaller $\mu$ implies that the system stabilizes more rapidly to the equilibrium point. According to Lemma \ref{lemma: lyapunov stable}, $\mathbb{E}[V(x_k,u_k+\delta v_k)]$ is upper-bounded by an exponentially decaying term and a constant. This follows from the fact that the increment of $\|u_k\|$ is bounded by $\frac{\eta}{2\delta}$, which is summable for an exponentially stable plant. Consequently, the system states remain bounded.
\end{remark}

To analyze optimality, we write
\begin{equation*}
    p_{u'}(u) = \mathbb{P}(\bm{1}_{u'\succ u}=1) = \sigma(\tilde{\Phi}(u)-\tilde{\Phi}(u'))
\end{equation*}
to simplify the notation. A smaller value of $\tilde{\Phi}(u)$ yields a lower $p_{u'}(u)$. This monotonic relationship allows the minimization of $\tilde{\Phi}(u)$ to be equivalently reformulated as the minimization of $p_{u'}(u)$, which is further supported by Lemma \ref{lemma: p}.
\begin{lemma}
\label{lemma: p}
    Let Assumptions \ref{ap: map lipschitz}-\ref{ap: user model} hold. $\forall u' \in \mathbb{R}^{n_u}$, $p_{u'}(u)$ is $L_{p,0}$-Lipschitz, and $L_{p,1}$-smooth with respect to $u$, where $L_{p,0}=L_{\sigma,0}L_{0}$, $L_{p,1}=\sigma'(0)L_{1}+L_{\sigma,1}L_{0}^2$. Furthermore, $p_{u'}(u)$ is partially convex with respect to $u$ if $\tilde{\Phi}(u)\leq\tilde{\Phi}(u')$.
\end{lemma}
\begin{proof}
    Proof on partially-convexity can be found in \cite{yue2009interactively}. The rest of the proof can be found in Appendix \ref{proof: p}.
\end{proof}
\begin{remark}
    Since $\sigma(t)$ is a bounded and smooth function, the Lipschitz continuity and smoothness properties are preserved for the composed function $p_{u'}(u)$. Convexity is also partially preserved since $\sigma(t)$ is partially convex. 
\end{remark}
Under the assumption of strong convexity, we show that, for any fixed $u'$, the minimizer of $p_{u'}(u)$ is unique and coincides with $u^*$.
\begin{lemma}
\label{lemma: unique solution}
    Let Assumptions \ref{ap: map lipschitz}-\ref{ap: user model} hold and $u^*$ be the unique solution of \eqref{eq: unconstrained objective}, then $\forall u' \in \mathbb{R}^{n_u}$, $u^*$ is also the unique minimizer of $p_{u'}(u)$.
\end{lemma}
\begin{proof}
    Proof can be found in Appendix \ref{proof: unique solution}.
\end{proof}

Based on Lemmas \ref{lemma: p} and \ref{lemma: unique solution}, together with standard techniques from convex analysis, we can show that gradient descent with respect to $p_{u'}(u)$
\begin{equation}
\label{eq: gradient descent in p}
    u_{k+1} = u_k-\eta \nabla p_{u_k}(u_k)
\end{equation} converges to $u^*$, where $\nabla p_{u_k}(u_k)$ denotes the gradient of $p_{u_k}(u)$ with respect to $u$ evaluated at $u_k$. Update \eqref{eq: gradient descent in p} corresponds to the ideal case in which the function $h$ is known. In other words, it can be shown that $\nabla \tilde{\Phi}(u_k)$ and $\nabla p_{u_k}(u_k)$ are proportional up to a constant scaling factor. Building on this evidence, we compactly write \eqref{eq: algorithm 2} as
\begin{equation}
\label{eq: error formulation}
    u_{k+1} = u_k-\eta(\nabla p_{u_k}(u_k) + e_k),
\end{equation}
where the error term $e_k$ is defined as
\begin{equation}
    e_k \!= \!-\frac{1}{2\delta}\bm{1}_{(x_{k\!+\!1},u_k \!+ \!\delta v_k)\succ (x_{k},u_{k\!-\!1}\!+\!\delta v_{k\!-\!1})}v_k\!-\!\nabla p_{u_k}(u_k).
\end{equation}
In \eqref{eq: error formulation}, we interpret \eqref{eq: algorithm 2} as an instance of \eqref{eq: gradient descent in p} with an additional error $e_k$. It exists because $h(u_k+\delta v_k)$ is approximated by $x_{k+1}$, and the gradient is estimated via random perturbations. The boundedness of $e_k$ is established in Lemma \ref{lemma: bounded error}.

\begin{lemma}
\label{lemma: bounded error}
    Let Assumptions \ref{ap: map lipschitz}-\ref{ap: user model} hold. $\|\mathbb{E}[e_k|\mathcal{F}_k]\|\leq \sqrt{R_1V(x_{k-1},u_{k-1}+\delta v_{k-1})+R_2}$, where $R_1 = \frac{2L_{\sigma,0}^2L_{x}^2(\mu+1)}{\alpha_2\delta^2}\mu$, and $R_2=\frac{2L_{\sigma,0}^2L_x^2a_1}{\alpha_2\delta^2}(2\delta^2+\eta+(\frac{\eta}{2\delta})^2)\mu+2a_2^2\delta^2$, $a_2=L_{p,1}\sqrt{n}+(\sigma'(0)L_{1}+L_{p,1})(1+\frac{\eta}{\delta})$
\end{lemma}
\begin{proof}
    Proof can be found in Appendix \ref{proof: bounded error}.
\end{proof}
In Lemma \ref{lemma: bounded error}, we observe that at time step $k$, conditioned on the natural filtration $\mathcal{F}_k$, the error $e_k$ is bounded by $R_1V(x_{k-1},u_{k-1}+\delta v_{k-1})$ and $R_2$, where $V(x_{k-1},u_{k-1}+\delta v_{k-1})$ is bounded as established in Lemma \ref{lemma: lyapunov stable}. Additionally, we observe that $R_1 = \mathcal{O}(\mu)$ and $R_2 = \mathcal{O}(\mu,\delta^2)$. Both of them decrease as $\mu$ decreases. This result is consistent with the approximation discussed in Section \ref{section: 3}, as a smaller $\mu$ enables $x_{k+1}$ to more accurately approximate $h(u_k+\delta v_k)$.

Now we are ready to present the main convergence result.
\begin{theorem}
\label{theorem: main result}
    $\forall k'$ and $\forall k> k'$, the expected distance to $u^*$ is bounded, i.e., $\mathbb{E}[\|u_{k}-u^*\|^2]\leq(\frac{1+\rho}{2})^{k-k'}\mathbb{E}[\|u_{k'}-u^*\|^2]+\mathcal{O}(\mu,\mu^{k'},\delta)$, where $\rho=1-2\sigma'(0)m\eta$.
\end{theorem}
\begin{proof}
    Proof can be found in Appendix \ref{proof: main result}.
\end{proof}
We adopt the expected distance to $u^*$, i.e., $\mathbb{E}[\|u_{k}-u^*\|^2]$, as the error metric, which is commonly used in the literature on online zeroth-order optimization \cite{wang2024online,tang2023zeroth}. Our results indicate that, for any fixed time step $k'$, $\mathbb{E}[\|u_{k}-u^*\|^2]$ is bounded by $\mathbb{E}[\|u_{k'}-u^*\|^2]$, scaled by an exponentially decaying factor, and a constant term of order $\mathcal{O}(\mu,\mu^{k'},\delta)$ that depends on $k'$. In steady state, $\mathbb{E}[\|u_{k'}-u^*\|^2]$ vanishes, so that the bound reduces to $\mathcal{O}(\mu,\mu^{k'},\delta)$. Since $\mu^{k'}$ decreases as $k'$ increases for $\mu<1$, the steady-state error is eventually fully characterized by $\mathcal{O}(\mu,\delta)$.
\section{Numerical simulation}
To demonstrate that the controller \eqref{eq: algorithm 2} is capable of identifying $u^*$, we perform numerical simulations on a Linear Time-Invariant (LTI) system defined by
\begin{equation}
    \begin{split}
        x_{k+1} = Ax_k + Bu_k,
    \end{split}
\end{equation}
where $A\in\mathbb{R}^{n_x\times n_x}$, $B\in\mathbb{R}^{n_x\times n_u}$ are the system matrices. This system has an invertible steady-state input-output map $H = (I-A)^{-1}B$ and is pre-stabilized by a lower-level controller, ensuring that the spectral radius of $A$ is less than one. Such systems frequently arise in applications, including building control \cite{lian2023adaptive}, and power systems \cite{zhao2015distributed}. 
\subsection{A simple example}
To analyze the impact of plant transients on the algorithm's performance, we first consider the quadratic problem
\begin{align*}
    \min&\quad(x-x_{\text{ref}})^{\top}(x-x_{\text{ref}}),\\
    s.t. &\quad x = Hu,
\end{align*}
where the solution is $u = H^{-1}x_{\text{ref}}$. We set
$A \!= \!\begin{bmatrix}
        c&1\\
        0&c
\end{bmatrix}$, $B$ to be the identity matrix and $x_{\text{ref}} = [100, 100]^{\top}$. The parameter $c$ varies between $c = 0.1$ and $c = 0.7$ to represent plants with different decay rates. The simulation parameters are chosen as $\eta=0.1$ and $\delta = 0.5$. The results are shown in Fig. \ref{fig: result for simple system}.

\begin{figure}[t!]
    \subfloat[c = 0.1.\label{fig: random 0.1}]{
       \includegraphics[width=0.47\linewidth]{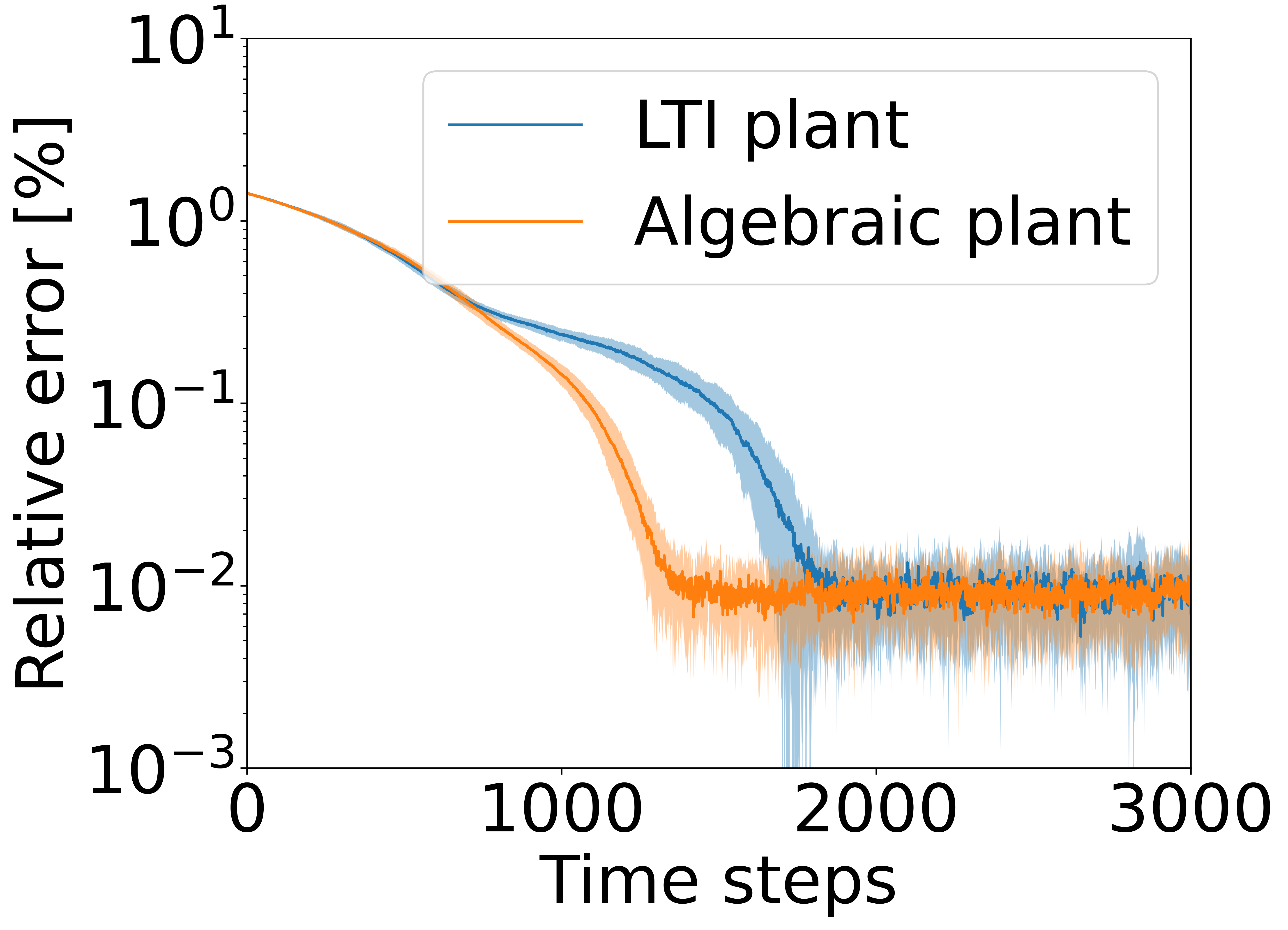}}
    \hfill
    \subfloat[c = 0.7.\label{fig: random 0.7}]{
        \includegraphics[width=0.47\linewidth]{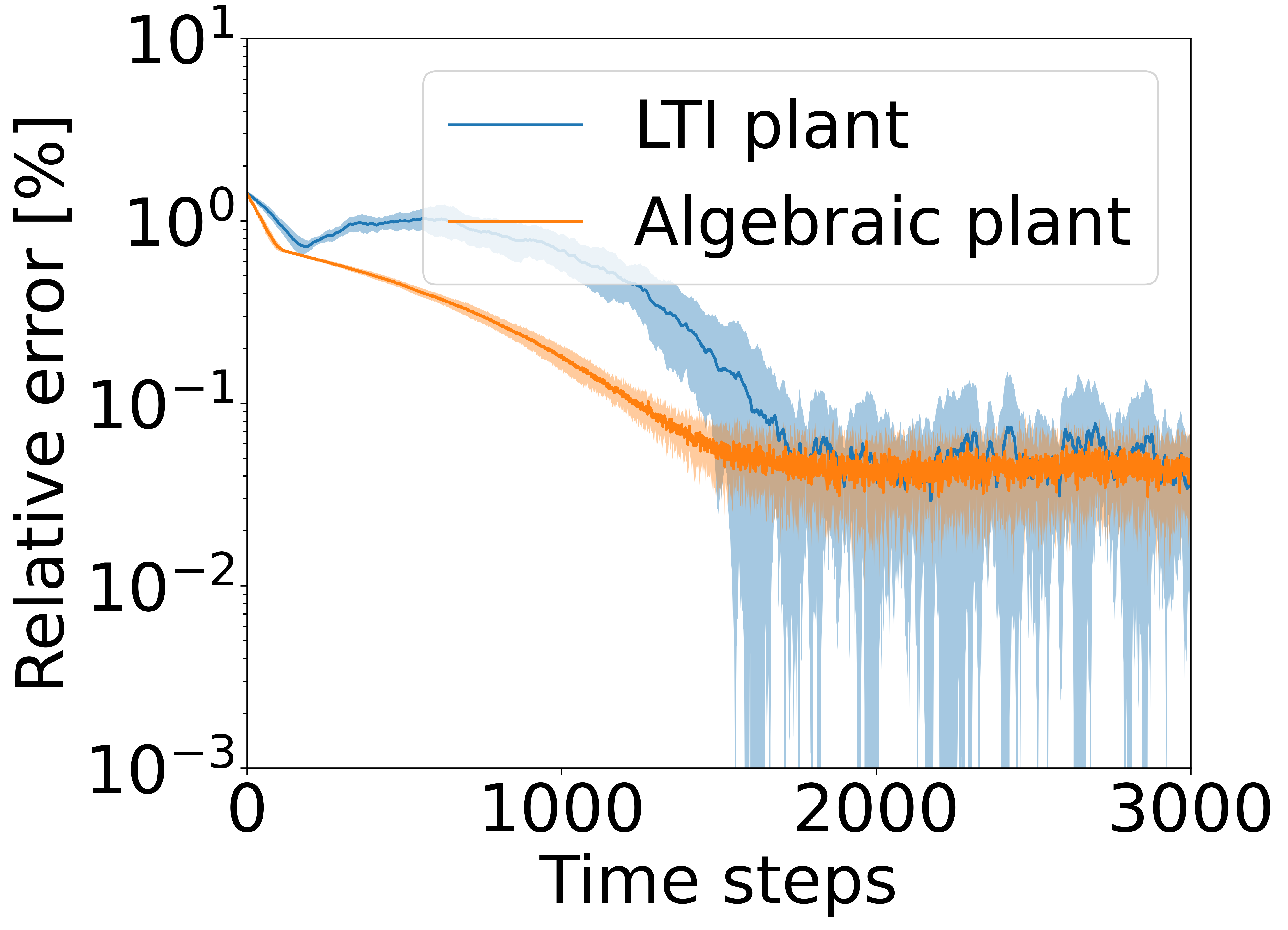}}
    \caption{Quadratic problem.}
    \label{fig: result for simple system}
\end{figure}

In Fig. \ref{fig: result for simple system}, the relative error $\|x_k-x_{\text{ref}}\|/\|x_{\text{ref}}\|$ on a logarithmic scale is shown for $c=0.1$ and $c=0.7$. The solid line represents the mean value over 20 simulations, while the shaded region indicates one standard deviation. The orange line corresponds to the algebraic plant, which assumes $H$ is known and samples $\bm{1}_{u_{k}+\delta v_k\succ u_{k-1}+\delta v_{k-1}}$ directly, whereas the blue line represents \eqref{eq: algorithm 2}. At the steady state, both \eqref{eq: algorithm 2} and its steady-state counterpart achieve a comparable level of accuracy. However, for the slower system ($c\!=\!0.7$), \eqref{eq: algorithm 2} exhibits a larger overshoot and higher steady-state variance. This behavior is expected, since the transient error becomes significant for slower systems, resulting in larger overshoot.

\subsection{Thermal comfort optimization}
Next, we consider a thermal comfort optimization problem, in which a building is represented by an LTI system with 13 states \cite{lian2023adaptive}. An occupant's thermal comfort utility is represented with the well-known Predictive Mean Vote (PMV) model \cite{fanger1970thermal}. The PMV output is typically expressed in terms of the Predicted Percentage of Dissatisfied (PPD) index, which is a nonconvex function of room temperature. Between any two indoor temperatures, the preference feedback is sampled according to Assumption \ref{ap: user model} where $\tilde{\Phi}$ is represented by PPD. The goal is to identify the room temperature that minimizes PPD using the controller \eqref{eq: algorithm 2}. For other PMV parameters, we assume the occupant is typing, wearing sweatpants, T-shirt and shoes or sandals (default setting in the PMV model).

In Fig. \ref{fig: result for building control}, the mean indoor temperature with \eqref{eq: algorithm 2} over 20 simulations (blue line) is plotted, while the orange line represents the case with a noise-free user model, i.e., $\bm{1}_{(x_{k+1},u_k+\delta v_{k})\succ (x_{k},u_{k-1}+\delta v_{k-1})} = \text{sign}(\Phi(x_{k},u_{k-1}+\delta v_{k-1})-\Phi(x_{k+1},u_k+\delta v_k))$. The shaded region represents one standard deviation, and the black horizontal line represents the true optimal temperature. With careful tuning of $\eta$ and $\delta$, controller \eqref{eq: algorithm 2} can track the optimal point effectively without large overshoot. which demonstrates its potential for learning a human's utility in real-world applications. Meanwhile, the algorithm exhibits a higher convergence rate with noise-free feedback. Quantifying its closed-loop behavior from a theoretical perspective remains an interesting direction for future work.

\begin{figure}[t!]
    \includegraphics[width=\linewidth]{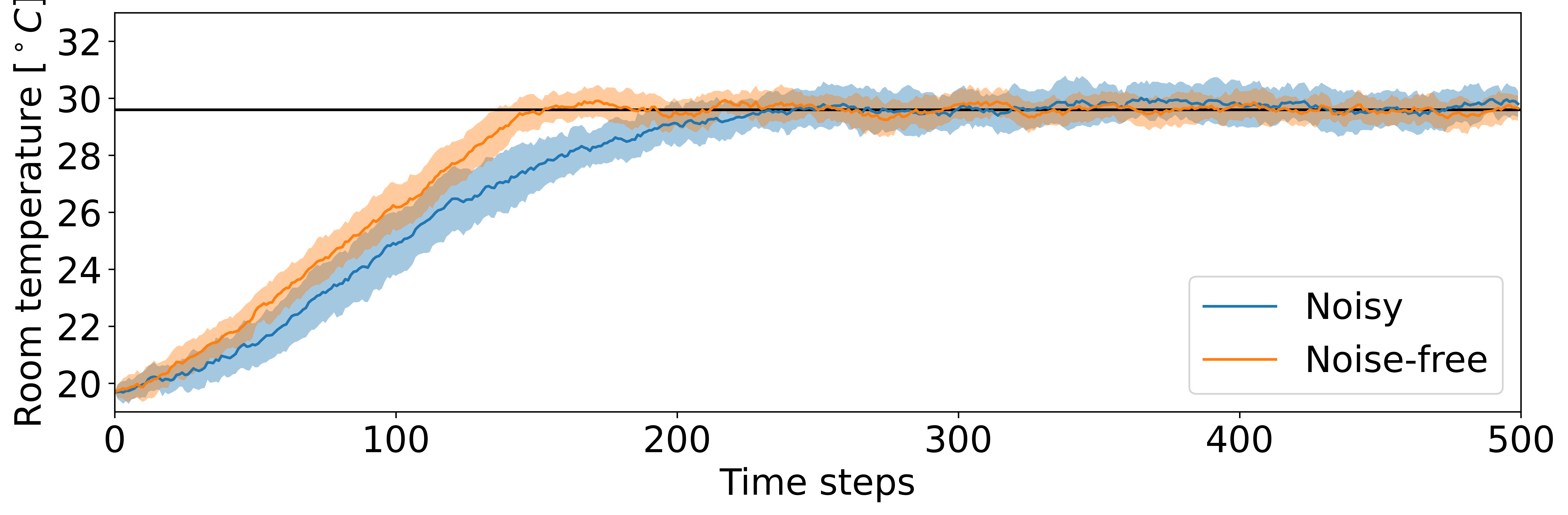}
    \caption{Thermal comfort optimization.}
    \label{fig: result for building control}
\end{figure}
\section{Conclusion}
In this work, we developed a human-aware controller that utilizes preference feedback to optimize utility while accounting for system transients. We derived an explicit upper bound on the error introduced by approximating the steady-state input-state map using the real-time state measurements. We analyzed the impact of both the system decay rate and smoothing parameters on the stability and optimality of the closed-loop system. Numerical experiments on a thermal comfort optimization task demonstrate its potential for solving real-world problems. Further research directions include extending the theoretical framework to alternative user models (e.g., noise-free model) and exploring real-world applications such as product design and chemical selection.

\bibliographystyle{IEEEtran}
\bibliography{references}

\appendix
\subsection{Proof of Lemma \ref{lemma: lyapunov stable}}
\label{proof: lyapunov stable}
\begin{proof}
We first consider $\mathbb{E}[V(x_k,u_k+\delta v_k)|\mathcal{F}_k]$, where $\mathcal{F}_k$ is the natural filtration at time step $k$.
    \begin{align*}
        \mathbb{E}[V(&x_k,u_k+\delta v_k)|\mathcal{F}_k]\\
        \leq&\mathbb{E}[\alpha_2\|x_k-h(u_{k-1}+\delta v_{k-1})\\
        &+h(u_{k-1}+\delta v_{k-1})-h(u_k+\delta v_{k})\|^2|\mathcal{F}_k]\\
        \leq&\mathbb{E}[2\alpha_2\|x_k-h(u_{k-1}+\delta v_{k-1})\|^2|\mathcal{F}_k]\\
        &+\mathbb{E}[2\alpha_2\|h(u_{k-1}+\delta v_{k-1})-h(u_k+\delta v_{k})\|^2|\mathcal{F}_k]\\
        \leq&2\frac{\alpha_2}{\alpha_1}V(x_{k},u_{k-1}+\delta v_{k-1})\\
        &+\mathbb{E}[2\alpha_2L_{h}^2\|u_{k-1}+\delta v_{k-1}-u_k-\delta v_{k}\|^2|\mathcal{F}_k]\\
        \leq&2\frac{\alpha_2}{\alpha_1}\!(\!1\!-\!\frac{\alpha_3}{\alpha_2}\!)\!V\!(\!x_{k\!-\!1},u_{k\!-\!1}\!+\!\delta v_{k\!-\!1}\!)\!+\!4\alpha_2L_{h}^2\!(\!\delta^2\!+\!(\!\delta\!+\!\frac{\eta}{2\delta}\!)^2)\\
        =&\mu V(x_{k-1},u_{k-1}+\delta v_{k-1})+a_1(2\delta^2+\eta+(\frac{\eta}{2\delta})^2).
    \end{align*}
    Apply this inequality $k$ times from $V(x_k,u_k+\delta v_k)$ to $V(x_0,u_0+\delta v_0)$, we obtain the desired result.
\end{proof}

\subsection{Proof of Lemma \ref{lemma: p}}
\label{proof: p}
\begin{proof}
For the Lipschitz property, we have $\forall u'$
    \begin{align*}
        |p_{u'}(u_1)-p_{u'}(u_2)|=&|\sigma\big(\tilde{\Phi}(u_1)\!-\!\tilde{\Phi}(u')\big)\!-\!\sigma\big(\tilde{\Phi}(u_2)\!-\!\tilde{\Phi}(u')\big)|\\
        \leq&L_{\sigma,0}L_{0}\|u_1-u_2\|.
    \end{align*}
For the smoothness property,
    \begin{align*}
        \|\nabla &p_{u'}(u_1)\!-\!\nabla p_{u'}(u_2)\|\\
        =&\|\sigma'(\tilde{\Phi}(u_1)\!-\!\tilde{\Phi}(u'))\nabla \tilde{\Phi}(u_1)\!-\!\sigma'(\tilde{\Phi}(u_1)\!-\!\tilde{\Phi}(u'))\nabla \tilde{\Phi}(u_2)\\
        &\!+\!\sigma'\!(\tilde{\Phi}(u_1)\!-\!\tilde{\Phi}(u'))\nabla \tilde{\Phi}(u_2)\!-\!\sigma'(\tilde{\Phi}(u_2)\!-\!\tilde{\Phi}(u'))\nabla \tilde{\Phi}(u_2)\|\\
        \leq& |\!\sigma'(\tilde{\Phi}(u_1)\!-\!\tilde{\Phi}(u'))\!|\!L_{1}\!\|u_1\!\!-\!u_2\!\|\!+\!L_{\sigma,1}L_{0}\|\!\nabla\tilde{\Phi}(u_2)\!\|\|\!u_1\!-\!u_2\!\|\\
        \leq&(\sigma'(0)L_{1}+L_{\sigma,1}L_{0}^2)\|u_1-u_2\|.\qedhere
    \end{align*}
\end{proof}

\subsection{Proof of Lemma \ref{lemma: unique solution}}
\label{proof: unique solution}
\begin{proof}
    Since $\sigma(t)$ is a strictly increasing function, $u^*$ that minimizes $\tilde{\Phi}(u)$ also minimizes $p_{u'}(u)$. Suppose, for contradiction, there exists two distinct minimizers of $p_{u'}(u)$, $u_1$ and $u_2$ such that $u_1\neq u_2$. Because $\sigma(t)$ is strictly increasing, it follows that $\tilde{\Phi}(u_1)=\tilde{\Phi}(u_2)$, implying that $\tilde{\Phi}(u)$ attains the minimum value at two different points. This contradicts the assumption that $\tilde{\Phi}(u)$ is strongly convex, which guarantees a unique minimizer. This concludes the proof.
\end{proof}

\subsection{Proof of Lemma \ref{lemma: bounded error}}
\label{proof: bounded error}
\begin{proof}
Let us consider $\|\mathbb{E}[e_k|\mathcal{F}_k]\|^2$.
    \begin{align*}
        \|\mathbb{E}[&e_k|\mathcal{F}_k]\|^2\\
        =&\|\frac{1}{\delta}\mathbb{E}[\sigma(\Phi(x_{k+1},u_{k}+\delta v_k)-\Phi(x_{k},u_{k-1}+\delta v_{k-1}))v_k\\
        &-\sigma(\Phi(h(u_{k}+\delta v_{k}),u_{k}+\delta v_k)\\
        &-\Phi(h(u_{k-1}+\delta v_{k-1}),u_{k-1}+\delta v_{k-1}))v_k\\
        &+\sigma(\Phi(h(u_{k}+\delta v_{k}),u_{k}+\delta v_k)\\
        &-\!\Phi(h(u_{k\!-\!1}\!+\!\delta v_{k\!-\!1}),u_{k\!-\!1}\!+\!\delta v_{k\!-\!1}))v_k|\mathcal{F}_k]\!-\!\nabla p_{u_k}(u_k)\|^2\\
        \leq&\frac{2}{\delta^2}\mathbb{E}[\|\sigma(\Phi(x_{k+1},u_{k}+\delta v_k)-\Phi(x_{k},u_{k-1}+\delta v_{k-1}))v_k\\
        &-\sigma(\Phi(h(u_{k}+\delta v_{k}),u_{k}+\delta v_k)\\
        &-\Phi(h(u_{k\!-\!1}\!+\!\delta v_{k\!-\!1}),u_{k\!-\!1}\!+\!\delta v_{k\!-\!1}))v_k\|^2|\mathcal{F}_k]\\
        &+2\|\nabla \tilde{p}_{u_{k-1}+\delta v_{k-1}}(u_k)-\nabla p_{u_{k-1}+\delta v_{k-1}}(u_k)\\
        &+\nabla p_{u_{k-1}+\delta v_{k-1}}(u_k)-\nabla p_{u_k}(u_k)\|^2\\
        \overset{\text{(s.1)}}{\leq}&\frac{4L_{\sigma,0}^2L_x^2}{\delta^2}\|x_{k}-h(u_{k-1}+\delta v_{k-1})\|^2\\
        &+\frac{4L_{\sigma,0}^2L_x^2}{\delta^2}\mathbb{E}[\|x_{k+1}-h(u_{k}+\delta v_{k})\|^2|\mathcal{F}_k]+2a_2^2\delta^2\\
        \leq&\frac{2L_{\sigma,0}^2L_x^2(\mu+1)}{\alpha_2\delta^2}\mu V(x_{k-1},u_{k-1}+\delta v_{k-1})\\
        &+\frac{2L_{\sigma,0}^2L_x^2a_1}{\alpha_2\delta^2}(2\delta^2+2\eta+(\frac{\eta}{\delta})^2)\mu+2a_2^2\delta^2,
    \end{align*}
    where in (s.1) we apply the bound from \cite[Lemma 1]{tang2023zeroth}, and perform analysis under the assumptions of smoothness and Lipschitz continuity.
\end{proof}

\subsection{Proof of Theorem \ref{theorem: main result}}
\label{proof: main result}
We first present a useful lemma. 
\begin{lemma}
\label{lemma: sequence}
    For non-negative sequence $a_k$ satisfying $\forall k, a_{k+1}^2\leq\rho a_{k}^2+b_ka_k+c$, where $\rho<1$, and $b_k$ is a non-increasing positive sequence, we have $\forall k'\geq 0$ and $\forall k> k', a_{k}^2\leq \rho'^{k-k'}a_{k'}^2+(a_{k'}^*)^2$, where $\rho' = (1-\frac{\sqrt{b_{k'}^2+4(1-\rho)c}}{2a_{k'}^*})$, $a_{k'}^*=\frac{b_{k'}+\sqrt{b_{k'}^2+4(1-\rho)c}}{2(1-\rho)}$.
\end{lemma}
\begin{proof}
$\forall k'\geq0$, if $a_i\geq a_{k'}^* \forall i\in \{k',\hdots, k-1\}$,
    \begin{align*}
        a_{i+1}^2\leq&\rho a_{i}^2+b_ia_i+c\leq(\rho-1) a_{i}^2+b_{k'}a_i+c+a_{i}^2\\
        \leq&(1-\frac{\sqrt{b_{k'}^2+4(1-\rho)c}}{2a_{k'}^*})a_{i}^2+\frac{\sqrt{b_{k'}^2+4(1-\rho)c}}{2}a_{k'}^*.
    \end{align*}
    Applying this inequality recursively, we have
        \begin{align*}
        a_{k}^2\leq&\rho'^{k-k'}a_{k'}^2+\frac{1}{1-\rho'}\frac{\sqrt{b_k^2+4(1-\rho)c}}{2}a_{k'}^*\\
        =&\rho'^{k-k'}a_{k'}^2 + (a_{k'}^*)^2.
    \end{align*}
    If $\exists i\in\{k',\hdots,k-1\}$ such that $a_i\leq a_{k'}^*$,
    \begin{align*}
        a_{i+1}^2\leq\rho a_{i}^2+b_ia_i+c\leq\rho a_{i}^2+b_{k'}a_i+c\leq(a_{k'}^*)^2.
    \end{align*}
    Since $\rho a_{i}^2+b_{k'}a_i+c$ is an increasing function when $a_i\geq 0$, we have $a_k^2\leq(a_{k'}^*)^2\leq\rho'^{k-k'}a_{k'}^2+(a_{k'}^*)^2$.
\end{proof}
Now, we are ready to present the main proof.
\begin{proof}
    \begin{align*}
    \mathbb{E}[\|&u_{k+1}-u^*\|^2|\mathcal{F}_k]\\
    =&\mathbb{E}[\|u_{k}\!-\!u^*\|^2\!-\!2\eta(\nabla p_{u_k}(u_k)\! + \!e_k)^{\top}(u_k\!-\!u^*)\! + \!\eta^2|\mathcal{F}_k]\\
    \leq&(1\!-\!2\sigma'(0)m\eta)\|u_{k}\!-\!u^*\|^2\!+\!2\eta\|\mathbb{E}[e_k|\mathcal{F}_k]\|\|u_k\!-\!u^*\|\! + \!\eta^2\\
    \leq&(1-2\sigma'(0)m\eta)\|u_{k}-u^*\|^2\\
    &+2\eta\sqrt{R_1V(x_{k-1},u_{k-1}\!+\!\delta v_{k-1})\!+\!R_2}\|u_k\!-\!u^*\|\! + \!\eta^2.
\end{align*}
\begin{align*}
    \mathbb{E}[\|&u_{k+1}-u^*\|^2]=\mathbb{E}[\mathbb{E}[\|u_{k+1}-u^*\|^2|\mathcal{F}_k]]\\
    \leq&(1-2\sigma'(0)m\eta)\mathbb{E}[\|u_{k}-u^*\|^2]\\
    &+2\eta\mathbb{E}[\sqrt{R_1V(x_{k-1},u_{k-1}\!+\!\delta v_{k-1})\!+\!R_2}\|u_k\!-\!u^*\|]\! + \!\eta^2\\
    \leq&(1-2\sigma'(0)m\eta)\mathbb{E}[\|u_{k}-u^*\|^2] + \eta^2\\
    &+\!2\eta\sqrt{\!R_1\!\mathbb{E}[V(x_{k\!-\!1},u_{k\!-\!1}\!+\!\delta v_{k\!-\!1})]\!+\!R_2}\sqrt{\mathbb{E}[\|u_k\!-\!u^*\|^2]},
\end{align*}
where $\mathbb{E}[V(x_{k\!-\!1},u_{k\!-\!1}\!+\!\delta v_{k\!-\!1})]\leq\mathbb{E}[V\!(x_0,u_0\!+\!\delta \!v_0)]\mu^{k\!-\!1}\!+\!\frac{a_1}{1\!-\!\mu}\!(\!2\delta^2\!+\!\eta\!+\!(\frac{\eta}{2\delta})^2)$.

Writing $\rho=1-2\sigma'(0)m\eta$, $b_k=2\eta\sqrt{b_1\mu^{k-1} + b_2}$, $c= \eta^2$ and applying Lemma \ref{lemma: sequence}, we have 
\begin{align*}
    \mathbb{E}[\|u_{k}&-u^*\|^2]\\
    \leq&\rho'^{k-k'}\mathbb{E}[\|u_{k'}-u^*\|^2] + (a_{k'}^*)^2\\
    \leq&(\frac{1\!+\!\rho}{2})^{k-k'}\mathbb{E}[\|u_{k'}\!-\!u^*\|^2]\!+\!\frac{b_1\mu^{k'-1}\! +\! b_2\! + \!2\sigma'(0)m\eta}{\sigma'(0)^2m^2},
\end{align*}
where $\rho' = \frac{b_{k'}+\rho\sqrt{b_{k'}^2+4(1-\rho)c}}{b_{k'}+\sqrt{b_{k'}^2+4(1-\rho)c}}$.
\end{proof}
\balance

\end{document}